\documentclass[12pt,reqno]{amsart}
\usepackage{amssymb,a4wide,xy}
\usepackage{amsmath,amsthm,amscd}
\usepackage{enumitem}
\usepackage{xcolor}
\xyoption{all}

\usepackage{multicol}
\usepackage{hyperref}

\newtheorem{theorem}[subsection]{Theorem}
\newtheorem{lemma}[subsection]{Lemma}
\newtheorem{proposition}[subsection]{Proposition}
\newtheorem{corollary}[subsection]{Corollary}

\theoremstyle{definition}
\newtheorem{remark}[subsection]{Remark}
\newtheorem{definition}[subsection]{Definition}
\newtheorem{example}[subsection]{Example}

\def\Lie{\operatorname{\textbf{\textsf{Lie}}}}
\def\Lb{\operatorname{\textbf{\textsf{Lb}}}}

\def\id{\operatorname{id}}

\def\de{\delta}

\def\pa{\partial}

\DeclareMathOperator{\Ker}{\mathsf {Ker}}

\DeclareMathOperator{\ab}{\rm{ab}}

\def\lim{\operatorname{lim}}

\newcommand{\mm}{\mathfrak{m}}
\newcommand{\nn}{\mathfrak{n}}
\newcommand{\Gg}{\mathfrak{g}}
\newcommand{\Aa}{\mathfrak{a}}

\begin{document}

\title[On capability of Leibniz algebras]{On capability of Leibniz algebras}

\author[E. Khmaladze]{Emzar Khmaladze$^1$}
\address{$^1$A. Razmadze Mathematical Institute of Tbilisi State University,
	Tamarashvili St. 6, 0177 Tbilisi, Georgia \& The University of Georgia, Kostava St. 77a, 0171 Tbilisi, Georgia  }
\email{e.khmal@gmail.com}
\author[R. Kurdiani]{Revaz Kurdiani$^2$}
\address{$^2$I. Javakhishvili Tbilisi State University,  University St. 13, TSU Building XI, 0186 Tbilisi, Georgia}
\email{revaz.kurdiani@tsu.ge}
\author[M. Ladra]{Manuel Ladra$^3$}
\address{$^3$Department of Mathematics, Universidade de Santiago de Compostela\\15782 Santiago de Compostela, Spain}
\email{manuel.ladra@usc.es}

\begin{abstract}
We study the capability property of Leibniz algebras via the non-abelian exterior product.
\end{abstract}

\subjclass[2010]{18G10, 18G50.}

\keywords{Leibniz algebra, capable Leibniz algebra, non-abelian tensor and exterior products, tensor and exterior centers.}

\maketitle

\section{Introduction}\label{S:In}

Leibniz algebra is a sort of non-commutative generalization of the Lie algebra structure, with non skew-symmetric bracket, where the Jacobi identity is replaced by the so called Leibniz identity. Leibniz algebras where firstly considered by Bloh in \cite{Blo} and later rediscovered by Loday in \cite{Lo}. This later was inspired by construction of a new (co)homology theory for Lie algebras, the so called Leibniz (co)homology (see \cite{Lo, LoPi}). Namely, Loday noticed that in the Lie homology complex the only property of the bracket needed was the Leibniz identity. Later, Leibniz algebra found other applications in different areas of mathematics and yet it is gaining increasing importance.

During the last 20 years lots of papers appeared investigating extensions of results from Lie to Leibniz algebras. Often in these investigations many non-obvious algebraic identities need to be verified and they are sufficiently non-trivial and interesting.

As an example of these generalizations, the non-abelian exterior product of Leibniz algebra was recently developed in \cite{DoGaKh} with various applications in Leibniz homology. This non-abelian exterior product is the quotient of the  non-abelian tensor product of Leibniz algebras developed by Gnedbaye \cite{Gn99} and at the same time, it is the non-commutative version of the non-abelian exterior product of Lie algebras introduced by Ellis \cite{Ell91}.

In this paper we choose to study capability of Leibniz algebras via the non-abelian exterior product of Leibniz algebras. Our decision was inspired by the known result of Ellis in \cite{Ell95, Ell98} saying that a group is capable if and only if its exterior center is trivial and by the similar investigations  for Lie algebras  realized in \cite{Ni1, Ni2}.

The paper is organized as follows. In Section \ref{S:Pre} we recall some basic definitions on Leibniz algebras and give necessary results for the development of the paper. In Section \ref{S:Ten} we present the definitions of the non-abelian tensor and exterior products slightly differently from the original ones in \cite{Gn99, DoGaKh} and introduce the notions of tensor and exterior centers of a Leibniz algebra. Main results are presented in Section \ref{S:main}. Here we show that a Leibniz algebra is capable if and only if its exterior center is trivial (Theorem \ref{T:main1}). We study relation between tensor and exterior centers and conclude that they coincide for any Leibniz algebra (Corollary \ref{C:main}). Finally, in Section \ref{S:lie} we establish capability condition for perfect Leibniz and Lie algebras (Proposition \ref{P:5.1}), we study tensor and exterior centers of products in Lie and Leibniz algebras (Propositions \ref{P:5.4} and \ref{P:5.5}) and show that there are Lie algebras capable in the category of Leibniz algebras, but not capable in the category of Lie algebras (Corollary \ref{C:last}).

\section{Leibniz algebras}\label{S:Pre}

Throughout the paper $\mathbb{K}$ is a field, unless otherwise stated. All vector spaces and algebras are considered over $\mathbb{K}$, all linear maps are $\mathbb{K}$-linear maps and $\otimes$ stands for $\otimes_{\mathbb{K}}$.

\begin{definition}[\cite{Lo}]
	A Leibniz algebra is a vector space $\mathfrak{g}$  equipped with a linear
	map (Leibniz bracket)
	\[
	[\;,\; ]\colon \mathfrak{g}\otimes \mathfrak{g} \to \mathfrak{g}
	\]
	satisfying the Leibniz identity
	\[
	[x,[y,z]] = [[x,y],z] - [[x,z],y],
	\]
	for all $x,y,z \in \mathfrak{g}$.
\end{definition}

A homomorphism of Leibniz algebras is a linear map preserving the bracket. The respective category of
Leibniz algebras will be denoted by $\Lb$.

The original reason to introduce Leibniz algebras was a new variant of Lie homology, called non-commutative Lie homology or Leibniz homology, developed in \cite{Lo0,LoPi} and denoted by $HL_*$.

Let us note that any Lie algebra is a Leibniz algebra and conversely, any Leibniz algebra with the antisymmetric Leibniz bracket is a Lie algebra. This is why Leibniz algebras are called non-commutative generalization of Lie algebras. Thus, there is a full embedding functor $\Lie \hookrightarrow \Lb$, where $\Lie$ denotes the category of Lie algebras. This embedding has a left adjoint $\mathfrak{{Lie}}\colon \Lb \to \Lie$, called the Liezation functor (see for example \cite{KuPi, CaKh}).

A subspace $\mathfrak{a}$ of a Leibniz algebra $\mathfrak{g}$  is called (two-sided) \emph{ideal} of $\mathfrak{g}$ if $[a,x],\;  [x,a]\in \mathfrak{a}$ for all $a\in \mathfrak{a}$ and $x\in \mathfrak{g}$. In this case the quotient space $\mathfrak{g}/\mathfrak{a}$ naturally inherits a Leibniz algebra structure.

An example of ideal of a Leibniz algebra $\mathfrak{g}$ is the \emph{commutator} of $\mathfrak{g}$, denoted by $[\mathfrak{g},\mathfrak{g}]$, which is the subspace of $\mathfrak{g}$  spanned by elements of the form $[x,y]$, $x,y\in \mathfrak{g}$. The quotient $\mathfrak{g}/[\mathfrak{g}, \mathfrak{g}]$ is denoted by $\mathfrak{g}^{\ab}$ and is called \emph{abelianization} of $\mathfrak{g}$. One more example of an ideal is the \emph{center} $Z(\mathfrak{g})= \{c \in {\mathfrak{g}} \mid [x,c] = 0=[c,x],\ \text{for all} \  x \in {\mathfrak{g}}\}$ of $\mathfrak{g}$. Note that both  $\mathfrak{g}^{\ab}$ and $Z(\mathfrak{g})$ are \emph{abelian Leibniz algebras}, that is, Leibniz algebras with the trivial Leibniz bracket.

\begin{definition} An epimorphism of Leibniz algebras $\mathfrak{f} \overset{\pi}{\twoheadrightarrow} \mathfrak{g}$ is called
\begin{enumerate}
\item[a)]\emph{a central extension} of Leibniz algebras, if $\Ker{(\pi)}\subseteq Z(\mathfrak{f})$;
\item[b)] \emph{a free presentation} of $\mathfrak{g}$, if $\mathfrak{f}$ is a free Leibniz algebra over a set.
\end{enumerate}
\end{definition}

The following lemma is straightforward.
\begin{lemma}\label{L:lamma1}
An epimorphism $\mathfrak{f} \overset{\pi}{\twoheadrightarrow} \mathfrak{g}$ of Leibniz algebras naturally induces a central extension of Leibniz algebras $\mathfrak{f}/[\Ker{(\pi)},\mathfrak{f}]\twoheadrightarrow \mathfrak{g}$.
\end{lemma}

In particular, if $0\longrightarrow \mathfrak{r}\longrightarrow \mathfrak{f} \overset{\pa}{\longrightarrow} \mathfrak{g}\longrightarrow 0$ is a free presentation of $\mathfrak{g}$, then $\mathfrak{f}/[\mathfrak{r},\mathfrak{f}]\overset{\pa}{\twoheadrightarrow} \mathfrak{g}$ will be called \emph{the central extension induced by this free presentation}.

\

\begin{definition}
We say that a Leibniz algebra $\mathfrak{g}$ is capable if there exists a Leibniz algebra $\mathfrak{q}$ and an isomorphism of Leibniz algebras
\[
\mathfrak{g}\cong\mathfrak{q}/Z(\mathfrak{q}).
\]
\end{definition}

Let us note that capability is defined in the same way in groups (resp. in Lie algebras). A group $G$ (resp. a Lie algebra $L$) is defined to be capable if there exist a group (resp. a Lie algebra) $Q$ such that the quotient of $Q$ by its center is isomorphic to $G$ (resp. $L$). It is well-known that the notion of capability in groups (resp. in Lie algebras) is related with inner automorphisms (resp. inner derivations). Namely,  a group (resp. a Lie algebra) is capable if and only if it is isomorphic to the group of the inner automorphisms of a group (resp. to the Lie algebra of inner derivations of a Lie algebra). To examine the similar fact for Leibniz algebras, we recall some notions from \cite{Lo} (see also \cite{CaKhLa}).

Let $\mathfrak{q}$ be a Leibniz algebra. A derivation $d$ (resp. an anti-derivation $\de$) of $\mathfrak{q}$ is a linear self-map $\mathfrak{q}\to \mathfrak{q}$
such that $d[x,y]=[d(x),y]+[x,d(y)]$ (resp. $\de[x,y]=[\de(x),y]-[\de(y),x]$) for all $x,y\in \mathfrak{q}$. For instance, given any $x\in \mathfrak{q}$, the map $d_{x}:\mathfrak{q}\to \mathfrak{q}$ given by $d_{x}(y)=[y,x]$, is a derivation and the map $\de_{x}:\mathfrak{q}\to \mathfrak{q}$ given by $\de_{x}(y)=-[x,y]$, is an anti-derivation.

The pair $(d,\de)$, where $d$ is a derivation and $\de$ is an anti-derivation of $\mathfrak{q}$ such that
\[
[x,d(y)]=[x, \de(y)] \quad \text{and} \quad \de[x,y]=-\de[y,x],
\]
is called a biderivation of $\mathfrak{q}$. For example, it is easy to check that for any $x\in \mathfrak{q}$ the pair $(d_{x},\de_{x})$ is  a biderivation, called the inner biderivation.
Denote by $\textsf{BiDer}(\mathfrak{q})$ the set of all biderivations of $\mathfrak{q}$. There is a Leibniz algebra structure on $\textsf{BiDer}(\mathfrak{q})$ (see \cite{Lo}) and it is called the Leibniz algebra of biderivations of $\mathfrak{q}$. The set of all inner biderivations, denoted by $\textsf{InnBiDer}(\mathfrak{q})$ is an ideal of $\textsf{BiDer}(\mathfrak{q})$, and it is called the Leibniz algebra of inner biderivations.

Using the same notations as above, it is easy to see that there is a homomorphism of Leibniz algebras $\mathfrak{q}\to \textsf{BiDer}(\mathfrak{q})$, $x\mapsto (d_{x},\de_{x})$, the image of which is $\textsf{InnBiDer}(\mathfrak{q})$ and the kernel is exactly the center $Z(\mathfrak{q})$. This homomorphism amounts  precisely that a Leibniz algebra $\mathfrak{q}$ is capable if and only if it occurs as the inner biderivations of some Leibniz algebra, that is, there is a Leibniz algebra $\mathfrak{q}$ such that
$\mathfrak{g}\cong \textsf{InnBiDer}(\mathfrak{q})$.

\section{Non-abelian tensor and exterior products of Leibniz algebras}\label{S:Ten}

\subsection{Leibniz actions and crossed modules.}
\begin{definition}
Let $\mm$ and $\nn$ be Leibniz algebras. A \emph{Leibniz action} of $\mm$ on
$\nn$ is a couple of bilinear maps $\mm \times \nn \to \nn$, $(m, n)\mapsto \;^mn$, and
$\nn \times \mm \to \nn$, $(n, m)\mapsto n^m$, satisfying the following axioms:
\begin{align*}
^{[m, m']}n&=\; ^m(^{m'}n)+(^mn)^{m'},& ^m[n, n']&=[^mn, n']-[^mn', n],\\
n^{[m, m']}&=(n^m)^{m'}-(n^{m'})^{m}, & [n, n']^m&=[n^m, n']+[n, n'^m],\\
^m(^{m'}n)&=-^m(n^{m'}), & [n, \;^mn']&=-[n, n'^{m}],
\end{align*}
for each $m, m'\in \mm$, $n, n'\in \nn$. For example, if $\mm$  is a subalgebra of a Leibniz algebra $\mathfrak{g}$  (maybe
  $\mathfrak{g}= \mathfrak{m}$) and $\nn$ is an ideal of $\mathfrak{g}$, then Leibniz bracket in $\mathfrak{g}$ yields a Leibniz action of $\mm$ on $\nn$.
\end{definition}


\begin{definition}
A \emph{Leibniz crossed module} $(\mm, \Gg, \eta)$ is a homomorphism of Leibniz algebras $\eta \colon \mathfrak{m} \to \mathfrak{g}$ together with an action of $\mathfrak{g}$ on $\mathfrak{m}$ such that
\begin{align*}
	& \eta(^xm)=[g,\eta(m)],   \qquad \eta(m^x)=[\eta(m),x],\\
	& ^{\eta(m_1)}m_2=[m_1,m_2]=m_1^{\eta (m_2)},
\end{align*}
where $x \in \Gg$ and $m, m_1, m_2 \in \mm$.
\end{definition}

\begin{example}\label{E:CM}
	Let $\mathfrak{a}$ be an ideal of a Leibniz algebra $\Gg$. Then the inclusion $i \colon \Aa \to \Gg$ is a crossed module where the action of $\Gg$ on $\Aa$ is given by the bracket in $\Gg$. In particular, a Leibniz algebra can be seen as a crossed module by $1_{\Gg} \colon \Gg \to \Gg$.
\end{example}

\subsection{Non-abelian tensor product}\label{S:tensor} The \emph{non-abelian tensor product}, $\mathfrak{m}\star \mathfrak{n}$, of two Leibniz algebras $\mathfrak{m}$ and $\mathfrak{n}$ with mutual action on each other, is introduced in \cite{Gn99}. For further convenience, we present the definition of $\mathfrak{m}\star \mathfrak{n}$ in a slightly different form, where the generators $m*n$ and $n*m$ from the original definition \cite[Definition-Theorem 4.1]{Gn99} are denoted by $m*_1n$ and $m*_2n$, respectively.

\begin{definition}\label{D:tensor}
Let $\mathfrak{m}$ and $\mathfrak{n}$ be Leibniz algebras acting on one another. The non-abelian tensor product of $\mathfrak{m}$ and $\mathfrak{n}$ is
the Leibniz algebra $\mathfrak{m}\star \mathfrak{n}$ generated by the symbols $m *_1n$ and $m*_2n$, for all $m\in \mathfrak{m}$ and $n\in \mathfrak{n}$, subject to the
following relations:
\begin{align*}
(1) \quad&k(m*_in)=km*_in=m*_ikn,  \\
(2\textrm{a}) \quad & (m+m')*_in=m*_in+m'*_in,   \\
(2\textrm{b}) \quad & m*_i(n+n')=m*_i n+m*_in', \\
(3\textrm{a})\quad&m*_1[n, n']=m^n*_1n'-m^{n'}*_1n, \\(3\textrm{b})\quad&[m,m']*_2 n=m'*_2n^m-m*_2n^{m'},\\
(3\textrm{c})\quad&[m, m']*_1n=m' *_2 {}^mn*-m*_1n^{m'}, \\ (3\textrm{d})\quad& m *_2[n, n']={}^nm*_1n'-m^{n'}*_2 n,\\
(4\textrm{a})\quad&m*_1 {}^{m'}n=-m*_1n^{m'}, \\ (4\textrm{b})\quad& {}^{n'}m*_2 n=-m {}^{n'}*_2 n,\\
(5\textrm{a})\quad& m^n*_1{}^{m'}n'=[m*_1n, m'*_1n']=m'^{\; n'}*_2{}^mn, \\ (5\textrm{b})\quad &^nm*_1n'^{\;m'}=[m*_2n, m'*_2n']={}^{n'}m'*_2n^m,\\
(5\textrm{c})\quad& m^n*_1 n'^{\;m'}=[m*_1n, m'*_2n']={}^{n'}m'*_2{}^mn, \\ (5\textrm{d})\quad&^nm*_1{}^{m'}n'=[m*_2n, m'*_1n']=m'^{\;n'}*_2 n^m,
\end{align*}
for $i=1, 2$ and $k\in \mathbb{K}$, $m, m'\in \mathfrak{m}$, $n, n'\in \mathfrak{n}$.
\end{definition}

There are induced homomorphisms of Leibniz algebras $\tau_{\mm} \colon \mm \star \nn \to \mm$ and $\tau_{\nn} \colon \mm \star \nn \to \nn$, where $\tau_{\mm}(m *_1 n) = m^n$, $\tau_{\mm}(m *_2 n) = {}^nm$ and $\tau_{\nn}(m *_1 n) = {}^mn$, $\tau_{\nn}(m *_2 n) = n^m$.

\subsection{Non-abelian exterior product}\label{exterior}
Let $\eta \colon \mm \to \Gg$ and $\mu \colon \nn \to \Gg$ be two Leibniz crossed modules. There are induced actions of $\mm$ and $\nn$ on each other via the action of $\Gg$. Therefore, we can consider the non-abelian tensor product $\mm \star \nn$. We define $\mm \square \nn$ as the vector subspace of $\mm \star \nn$ generated by the elements $m *_1 n' - m' *_2 n$ such that $\eta(m) = \mu(n)$ and $\eta(m') = \mu(n')$. It is shown in \cite[Proposition 1]{DoGaKh} that $\mm \square \nn$ is an ideal of $\mm \star \nn$.
\begin{definition}\cite{DoGaKh}
	Let $\eta \colon \mm \to \Gg$ and $\mu \colon \nn \to \Gg$ be two Leibniz crossed modules. The \emph{non-abelian exterior product} $\mm \curlywedge \nn$ of $\mm$ and $\nn$ is defined to be
	\[
	\mm \curlywedge \nn = \dfrac{\mm \star \nn} {\mm \square \nn}.
	\]
	The cosets of $m *_1 n$ and $m *_2 n$ will be denoted by $m \curlywedge_1 n$ and $m \curlywedge_2 n$, respectively.
\end{definition}

There is an epimorphism of Leibniz algebras $\pi\colon \mathfrak{m}\star \mathfrak{n} \to \mathfrak{m}\curlywedge \mathfrak{n}$ sending $m *_1 n$ and $m *_2 n$ to $m \curlywedge_1 n$ and $m \curlywedge_2 n$, respectively.

Let $\mathfrak{a}$ and $\mathfrak{b}$ be two ideals of a Leibniz algebra $\mathfrak{g}$. Then they act on each other via the Leibniz bracket in $\mathfrak{g}$. Let us note that, for each $a\in \mathfrak{a}$ and $b\in \mathfrak{b}$, the non-abelian tensor product $\mathfrak{a}\star \mathfrak{b}$ has two types of generators: $a *_1 b$ and $a *_2 b$.  Further, the non-abelian exterior product $\mathfrak{a}\curlywedge \mathfrak{b}$ (defined via the inclusion crossed modules $\mathfrak{a}\hookrightarrow \mathfrak{g}$ and $\mathfrak{b}\hookrightarrow \mathfrak{g}$) is just the quotient of $\mathfrak{a}\star \mathfrak{b}$ by the elements of the form $c *_1 c' - c' *_2 c$, where $c, c' \in \mathfrak{a}\cap \mathfrak{b}$. In particular, $\mathfrak{g}\curlywedge \mathfrak{g}$ is the quotient of $\mathfrak{g}\star \mathfrak{g}$ by the relation $x*_1y=y*_2x$ for  $x, y\in \mathfrak{g}$.

The following lemma is straightforward
\begin{lemma}\label{L:lemma ex ten}
If $\mathfrak{g}$ is an abelian Leibniz algebra,
then there is an isomorphism of vector spaces
\[
\mathfrak{g}\curlywedge \mathfrak{g}\cong \mathfrak{g}\otimes \mathfrak{g},
\]
where $\mathfrak{g}\otimes \mathfrak{g}$ denotes the tensor product of vector spaces.
\end{lemma}
The following proposition is immediate.
\begin{proposition}\cite{DoGaKh}
	Let $\mathfrak{a}$ and $\mathfrak{b}$ be two ideals of a Leibniz algebra. There is a homomorphism of Leibniz algebras
	\[
	\theta_{\mathfrak{a}, \mathfrak{b}} \colon 	\mathfrak{a}\star \mathfrak{b} \to \mathfrak{a}\cap \mathfrak{b} \quad (\text{resp.} \ \theta_{\mathfrak{a}, \mathfrak{b}} \colon 	\mathfrak{a}\curlywedge \mathfrak{b} \to \mathfrak{a}\cap \mathfrak{b})
	\]
	defined on generators by
	$\theta_{\mathfrak{a}, \mathfrak{b}}(a\star_1 b)= [a, b]$ and $\theta_{\mathfrak{a}, \mathfrak{b}}(a\star_2 b)=[b, a]$ (resp. $\theta_{\mathfrak{a}, \mathfrak{b}}(a\curlywedge_1 b)= [a, b]$ and $\theta_{\mathfrak{a}, \mathfrak{b}}(a\curlywedge_2 b)=[b, a]$),
	for all $a\in \mathfrak{a}$ and $b\in \mathfrak{b}$.
\end{proposition}

\begin{proposition}\cite{DoGaKh}\label{P:perfect}
	Let $\mathfrak{g}$ be a perfect Leibniz algebra, that is, $\mathfrak{g}=[\mathfrak{g},\mathfrak{g}]$. {Then $\mathfrak{g}\star\mathfrak{g}=\mathfrak{g}\curlywedge \mathfrak{g}$ and
	the homomorphism $\theta_{\mathfrak{g}, \mathfrak{g}} \colon \mathfrak{g}\curlywedge \mathfrak{g} \twoheadrightarrow \mathfrak{g}$ is the universal central extension of $\mathfrak{g}$.}
\end{proposition}
\begin{proof}The last four identities of the non-abelian tensor product immediately imply that $\mathfrak{g}\star\mathfrak{g}=\mathfrak{g}\curlywedge \mathfrak{g}$. Hence,  by \cite[Theorem 6.5]{Gn99} the homomorphism $\theta_{\mathfrak{g}, \mathfrak{g}} \colon \mathfrak{g}\curlywedge \mathfrak{g} \twoheadrightarrow \mathfrak{g}$ is the universal central extension of the perfect Leibniz algebra $\mathfrak{g}$.
\end{proof}

\subsection{Tensor and exterior centers}\label{centers} The following notions are useful in the study of capability of Leibniz algebras.

\begin{definition}
Let $\mathfrak{g}$ be a Leibniz algebra.
\begin{enumerate}
\item[a)] The tensor center $Z^{\star}(\mathfrak{g})$ of  $\mathfrak{g}$ is defined to be
\[
Z^{\star}(\mathfrak{g})=\{g\in \mathfrak{g} \mid g \star_1 x = g \star_2 x =0 \ \text{for all} \ x\in \mathfrak{g} \};
\]
\item[b)] The exterior center $Z^{\wedge}(\mathfrak{g})$ of  $\mathfrak{g}$ is defined to be
\[
Z^{\curlywedge}(\mathfrak{g})=\{g\in \mathfrak{g} \mid g \curlywedge_1 x = g \curlywedge_2 x =0 \ \text{for all} \ x\in \mathfrak{g} \}.
\]
\end{enumerate}
\end{definition}

\begin{lemma}
For any Leibniz algebra $\mathfrak{g}$ both $Z^{\star}(\mathfrak{g})$ and $Z^{\curlywedge}(\mathfrak{g})$ are ideals of $\mathfrak{g}$ contained in the center  $Z(\mathfrak{g})$.
\end{lemma}
\begin{proof} Thanks to the identities (1), (2a), (2b) in Definition \ref{D:tensor}  $Z^{\star}(\mathfrak{g})$ is a vector subspace of $\mathfrak{g}$ and it is an ideal because of the identities (3a)-(3d). For any $g\in Z^{\star}(\mathfrak{g})$ and $x\in \mathfrak{g} $ we have
\begin{align*}
&[g,x]=\theta_{\mathfrak{g}, \mathfrak{g}}(g\star_1 x)=0, \\
& [x,g]=\theta_{\mathfrak{g}, \mathfrak{g}}(g\star_2 x)=0,
\end{align*}
which means that $Z^{\star}(\mathfrak{g})\subseteq Z(\mathfrak{g})$. Similarly $Z^{\curlywedge}(\mathfrak{g})$ is an ideal of $\mathfrak{g}$ contained in $Z(\mathfrak{g})$.
\end{proof}

It is clear that $Z^{\star}(\mathfrak{g}) \subseteq Z^{\curlywedge}(\mathfrak{g})$. It follows by Proposition \ref{P:perfect} that this inclusion is an equality if $\mathfrak{g}$ is a perfect Leibniz algebra. We will show later that $Z^{\star}(\mathfrak{g}) = Z^{\curlywedge}(\mathfrak{g})$ for any Leibniz algebra $\mathfrak{g}$.

\section{Main results}\label{S:main}

\subsection{Capability condition via exterior center}

\begin{lemma}
Let $\mathfrak{g}$ a Leibniz algebra and $\mathfrak{c} \overset{\pa^*}{\longrightarrow} \mathfrak{g}$ be the central extension induced by a free presentation $0\longrightarrow \mathfrak{r}\longrightarrow \mathfrak{f} \overset{\pa}{\longrightarrow} \mathfrak{g}\longrightarrow 0$ of $\mathfrak{g}$. Then there is a natural isomorphism of Leibniz algebras
\[
[\mathfrak{c}, \mathfrak{c}]\overset{\cong}{\longrightarrow} \mathfrak{g} \curlywedge \mathfrak{g}, \quad [c_1, c_2]\mapsto \pa(c_1)\curlywedge \pa(c_2)
\]
\end{lemma}\label{L:lemma4.2}
\begin{proof}
As we know $\mathfrak{c}=\mathfrak{f}/[\mathfrak{f}, \mathfrak{r}]$ (see Lemma \ref{L:lamma1}). Then it is easy to see that there is a natural isomorphism $[\mathfrak{c}, \mathfrak{c}]\cong [\mathfrak{f}, \mathfrak{f}]/[\mathfrak{r}, \mathfrak{f}]$. By \cite[Lemma 1]{DoGaKh} we have the exact sequence of Leibniz algebras $\mathfrak{r}\curlywedge\mathfrak{f}\overset{i}{\longrightarrow} \mathfrak{f}\curlywedge\mathfrak{f}\longrightarrow \mathfrak{g}\curlywedge\mathfrak{g}\longrightarrow 0$ and so $(\mathfrak{f}\curlywedge\mathfrak{f})/{i(\mathfrak{r}\curlywedge\mathfrak{f})}\cong \mathfrak{g}\curlywedge\mathfrak{g}$. On the other hand, since  $\mathfrak{f}$ is a free Leibniz algebra, it follows from \cite[Proposition 4]{DoGaKh} that we have the isomorphism  $\theta_{\mathfrak{f}, \mathfrak{f}}: \mathfrak{f}\curlywedge\mathfrak{f}\overset{\cong}{\longrightarrow} [\mathfrak{f}, \mathfrak{f}] $
  and $\theta_{\mathfrak{f}, \mathfrak{f}} i (\mathfrak{r}\curlywedge\mathfrak{f})= [\mathfrak{r}, \mathfrak{f}] $. Then we get
\[
[\mathfrak{c}, \mathfrak{c}]\cong [\mathfrak{f}, \mathfrak{f}]/[\mathfrak{r}, \mathfrak{f}] \cong (\mathfrak{f}\curlywedge\mathfrak{f})/{i(\mathfrak{r}\curlywedge\mathfrak{f})} \cong
\mathfrak{g}\curlywedge\mathfrak{g}.
\]
\end{proof}

\begin{corollary}\label{C:coro}
Let $\mathfrak{c} \overset{\pa^*}{\longrightarrow} \mathfrak{g}$ be the central extension induced by a free presentation of $\mathfrak{g}$. Then
$x\in Z(\mathfrak{c})$ if and only if $\pa^*(x)\in Z^{\curlywedge}(\mathfrak{g})$.

\end{corollary}
\begin{proof}
If $x\in Z(\mathfrak{c})$,  i.e. $[x, c]=[c,x]=0$ for all $c\in \mathfrak{c}$, since $\pa$ is an epimorphism, thanks to Lemma \ref{L:lemma4.2}  we get  $\pa^*(x)\curlywedge_1\pa^*(c)=\pa^*(x)\curlywedge_2\pa^*(c)=0$, i. e. $\pa^*(x)\in Z^{\curlywedge}(\mathfrak{g})$.

Conversely, if $\pa^*(x)\in Z^{\curlywedge}(\mathfrak{g})$,  since  the isomorphic images in $\mathfrak{g} \curlywedge \mathfrak{g}$  of both $[x,c], [c,x]\in [\mathfrak{c},\mathfrak{c}]$ are zero, then  $[x,c]= [c,x]=0$ for every $c\in \mathfrak{c}$. Thus $x\in Z(\mathfrak{c})$.
\end{proof}

We have proved that $\pa^*(Z(\mathfrak{c}))\subseteq Z^{\curlywedge}(\mathfrak{g})$ and the preimage of $Z^{\curlywedge}(\mathfrak{g})$ is $Z(\mathfrak{c})$, this means that
\begin{equation}\label{E:eq1}
\pa^*(Z(\mathfrak{c}))= Z^{\curlywedge}(\mathfrak{g}).
\end{equation}

\begin{theorem}\label{T:main1}
A Leibniz algebra $\mathfrak{g}$ is capable if and only if $Z^{\curlywedge}(\mathfrak{g})=0$.
\end{theorem}
\begin{proof}
First suppose that $Z^{\curlywedge}(\mathfrak{g})=0$. Consider any free presentation  $0\to \mathfrak{r}\to \mathfrak{f} \to \mathfrak{g}\to 0$ of $\mathfrak{g}$ and its induced central extension $\mathfrak{c} \overset{\pa}{\longrightarrow} \mathfrak{g}$. It suffices to show $\Ker \pa =Z(\mathfrak{c}) $. We have $\Ker \pa \subseteq Z(\mathfrak{c}) $ because of the centrality. To see the inverse inclusion we note that given any $x\in Z(\mathfrak{c})$, by Corollary  \ref{C:coro} $\pa(x)\in Z^{\curlywedge}(\mathfrak{g})=0$ and so $x\in \Ker \pa.$

Conversely, let $\mathfrak{g}$ be capable. Then there is a Leibniz algebra $\mathfrak{q}$ such that $\mathfrak{q}/Z(\mathfrak{q})\cong \mathfrak{g}$. In other words, there is an epimorphism of Leibniz algebras $\pi:\mathfrak{q}\twoheadrightarrow \mathfrak{g}$ such that $\Ker \de =Z(\mathfrak{q})$.  Consider a free presentation  $0\to \mathfrak{s}\to \mathfrak{f} \overset{\tau}{\to} \mathfrak{q}\to 0$ of $\mathfrak{q}$. It implies a free presentation
$0\to \mathfrak{r}\to \mathfrak{f} \overset{\pa}{\to} \mathfrak{g}\to 0$ of $\mathfrak{g}$, where $\pa= \pi\tau$ and the kernel of $\pa$ is denoted by $\mathfrak{r}$. All these data give the following commutative diagram with exact rows

\begin{equation*}
\xymatrix@+10pt{
0\ \ar@{->}[r]& \mathfrak{r} \ \ar@{->}[r]
\ar@{->}[d]_{ }
 &\mathfrak{f}\ar@{->}[r]^{\pa}
\ar@{->}[d]_{\tau }
&\mathfrak{g} \ar@{->}[r]
\ar@{=}[d]_{ }
&0 \\
0\ \ar@{->}[r]
& Z(\mathfrak{q}) \ar@{->}[r]
 & \mathfrak{q} \ar@{->}[r]^{\pi}
& \mathfrak{g} \ar@{->}[r]
&0 .
}\end{equation*}
Now we claim that $\tau[\mathfrak{r},\mathfrak{f}]=0$. Indeed, given any $r\in \mathfrak{r}$ and $f\in \mathfrak{f}$ we have $\tau[r,f]=[\tau(r), \tau(f)]=0$, since $\tau(r)\in Z(Q) =\Ker\pi$. Let us denote $\mathfrak{f}/[\mathfrak{r}, \mathfrak{f}]$ by $\mathfrak{c}$. Then we have the following induced commutative triangle of Leibniz algebras

 \begin{equation*}
\xymatrix@+10pt{
& \mathfrak{c}\ar@{->}[rd]^{\pa^{*}}
\ar@{->}[rr]^{\tau^{*} }
& &\mathfrak{q} \ar@{->}[ld]_{\pi}
 \\
& &   \mathfrak{q} &
}
\end{equation*}
Since $\tau^{*}$ is an epimorphism, then $\tau^{*}(Z(\mathfrak{c}))\subseteq Z(\mathfrak{q})= \Ker \pi $, therefore $\pa^{*}(Z(\mathfrak{c}))=0$, i.e. $Z(\mathfrak{c})\subseteq \Ker \pa^{*}$. On the other hand, as we know $\pa^{*}:\mathfrak{c}\to \mathfrak{c}$ is a central extension and so $\Ker \pa^{*} \subseteq Z(\mathfrak{c})$. Thus $Z(\mathfrak{c})= \Ker \pa^{*}$. Then, by using the equality (\ref{E:eq1}) we get
\[
Z^{\curlywedge}(\mathfrak{g})=\pa^{*}(Z(\mathfrak{c}))=\pa^{*}(\Ker \pa^{*})=0.
\]
\end{proof}

\subsection{Coincidence of tensor and exterior  centers} In this subsection we will use some ideas from \cite{AnDoSiTh, BlFuMo, Ni2}

Given a Leibniz algebra $\mathfrak{g}$ we define
\[
\nabla(\mathfrak{g})=\Ker\left({\mathfrak{g}\star\mathfrak{g}\twoheadrightarrow \mathfrak{g}\curlywedge\mathfrak{g}}\right).
\]

\begin{lemma}\label{L:nabla}
For any Leibniz algebra $\mathfrak{g}$ there is an isomorphism
\[
\nabla(\mathfrak{g})\cong \nabla(\mathfrak{g}^{ab})
\]
\end{lemma}
\begin{proof}
As an easy consequence of \cite[Proposition 14]{DoGaKh} we have commutative diagram with exact rows
\begin{equation*}
\xymatrix@+10pt{
0\ \ar@{->}[r]& \Gamma(\mathfrak{g}^{ab}) \ \ar@{->}[r]
\ar@{=}[d]_{ }
 &\nabla(\mathfrak{g})\ar@{->}[r]
\ar@{->}[d]
&\mathfrak{g}^{ab}\curlywedge \mathfrak{g}^{ab}\ar@{->}[r]
\ar@{=}[d]_{ }
&0 \\
0\ \ar@{->}[r]
& \Gamma(\mathfrak{g}^{ab}) \ar@{->}[r]
  &\nabla(\mathfrak{g}^{ab}) \ar@{->}[r]
& \mathfrak{g}^{ab}\curlywedge \mathfrak{g}^{ab} \ar@{->}[r]
&0 ,
}\end{equation*}
where $\Gamma$ is the universal quadratic functor \cite{SiTy}. Then the required isomorphism follows.
\end{proof}

\begin{corollary}
Let $\pi: \mathfrak{g}\star \mathfrak{g}\twoheadrightarrow \mathfrak{g}^{ab}\star \mathfrak{g}^{ab}$ be the natural epimorphism. Then we have
\[
\Ker\pi\cap \nabla(\mathfrak{g})= \{0\}.
\]
\end{corollary}

\begin{proposition}\label{P:pr}
Given a Leibniz algebra $\mathfrak{g}$, there is an isomorphism of Leibniz algebras
\begin{equation}\label{E:eq2}
\mathfrak{g}\star \mathfrak{g}\cong (\mathfrak{g} \curlywedge\mathfrak{g})\times \nabla(\mathfrak{g}).
\end{equation}
\end{proposition}
\begin{proof}
It follows easily by defining relation (5a)-(5d) of the non-abelian tensor product in Definition \ref{D:tensor} that $\mathfrak{g}\star \mathfrak{g}$, $\mathfrak{g} \curlywedge\mathfrak{g}$ and $\nabla(\mathfrak{g})$ are abelian Leibniz algebras, i. e. just vector spaces. So the equality (\ref{E:eq2}) holds in this case.

Now, given any not necessarily abelian Leibniz algebra $\mathfrak{g}$, consider the following commutative diagram of Leibniz algebras with exact rows
\begin{equation*}
\xymatrix@+10pt{
0\ \ar@{->}[r]& \nabla(\mathfrak{g}) \ \ar@{->}[r]^{i}
\ar@{->}[d]_{\cong }
 &\mathfrak{g}\star \mathfrak{g}\ar@{->}[r]
\ar@{->}[d]
&\mathfrak{g}\curlywedge \mathfrak{g}\ar@{->}[r]
\ar@{=}[d]_{ }
&0 \\
0\ \ar@{->}[r]
& \nabla(\mathfrak{g}^{ab}) \ar@{->}[r]^{j}
  &\mathfrak{g}^{ab}\star \mathfrak{g}^{ab}  \ar@{->}[r]
& \mathfrak{g}^{ab}\curlywedge \mathfrak{g}^{ab} \ar@{->}[r]
&0.
}\end{equation*}
The bottom row is the exact sequence of abelian Leibniz algebras, i. e. just vector spaces. Then the injection $j$ has the left inverse $j':\mathfrak{g}^{ab}\star \mathfrak{g}^{ab}\to \nabla(\mathfrak{g}^{ab})$, that is $j' j=\id_{\nabla(\mathfrak{g}^{ab})}$. Since  the left vertical map is an isomorphism by Lemma \ref{L:nabla}, there is a left inverse homomorphism of $i$ defined by the composition three homomorphisms: the middle vertical map, $j'$ and the inverse of the left vertical map. Then the required isomorphism follows.
\end{proof}

\begin{theorem}\label{T:t_e_centers}
For any Leibniz algebra $\mathfrak{g}$ we have
\[
Z^{\star}(\mathfrak{g})=Z^{\curlywedge}(\mathfrak{g})\cap [\mathfrak{g},\mathfrak{g}].
\]
\end{theorem}
\begin{proof}
Take $x\in Z^{\curlywedge}(\mathfrak{g})\cap [\mathfrak{g},\mathfrak{g}]$. Then  $x\in Z^{\curlywedge}(\mathfrak{g})$ implies that both $g\star_1 x, \; g\star_2 x\in \nabla{(\mathfrak{g})}$ and   $x\in [\mathfrak{g},\mathfrak{g}]$ implies that
$g\star_1 x, \; g\star_2 x\in \Ker(\pi: \mathfrak{g}\star \mathfrak{g}\twoheadrightarrow \mathfrak{g}^{ab}\star \mathfrak{g}^{ab} )$ for any $g\in \mathfrak{g}$, i.e. $g\star_1 x, \; g\star_2 x\in \nabla{(\mathfrak{g})}\cap \Ker\pi$.  Since $\nabla{(\mathfrak{g})}\cap \Ker\pi =\{0\}$ by Corollary \ref{C:coro}, we have that $x\in Z^{\star}(\mathfrak{g})$. Thus  $ Z^{\curlywedge}(\mathfrak{g})\cap [\mathfrak{g},\mathfrak{g}]  \subseteq Z^{\star}(\mathfrak{g})$.

For the inverse inclusion, since $Z^{\star}(\mathfrak{g})\subseteq Z^{\curlywedge}(\mathfrak{g})$, it suffices to show that  $Z^{\star}(\mathfrak{g})\subseteq [\mathfrak{g}, \mathfrak{g}]$. Take $x\in Z^{\star}(\mathfrak{g})$, that is $x\in \mathfrak{g}$ such that $x\star_1 g=x\star_2 g=0$ for all $g\in \mathfrak{g}$. In particular, $x\star_1 x=x\star_2 x=0$ in $\mathfrak{g}\star \mathfrak{g}$ and
$\overline{x}\star_1 \overline{x}=\overline{x}\star_2 \overline{x}=0$ in $g^{ab}\star \mathfrak{g}^{ab}$, where $\overline{x}$ denotes the coset $x+[\mathfrak{g}, \mathfrak{g}]$. If we assume $x\notin [\mathfrak{g}, \mathfrak{g}]$ then we can deduce that $\overline{x}\otimes \overline{x}\neq 0$ in  $\mathfrak{g}^{ab}\otimes \mathfrak{g}^{ab}$. Indeed, by considering a bases of the vector space $\mathfrak{g}$ containing $x$, then the corresponding bases of $\mathfrak{g}^{ab}\otimes \mathfrak{g}^{ab}$ contains  $\overline{x}\otimes \overline{x}$. This contradicts to the natural isomorphism
\[
\mathfrak{g}^{ab}\star \mathfrak{g}^{ab}\cong \mathfrak{g}^{ab}\otimes \mathfrak{g}^{ab}\oplus \mathfrak{g}^{ab}\otimes \mathfrak{g}^{ab},
\]
which is given in \cite[Proposition 4.2]{Gn99} in a more general setting.
\end{proof}

\begin{corollary}\label{C:main}
For any Leibniz algebra $\mathfrak{g}$ we have
\[
Z^{\star}(\mathfrak{g})=Z^{\curlywedge}(\mathfrak{g}).
\]
\end{corollary}
\begin{proof}
By using Theorem \ref{T:t_e_centers} it suffices to show that $Z^{\curlywedge}(\mathfrak{g})\subseteq [\mathfrak{g},\mathfrak{g}]$. Take $x\in Z^{\curlywedge}(\mathfrak{g})$, that is $x\in \mathfrak{g}$ such that $x\curlywedge_1 g=x\curlywedge_2 g=0$ for all $g\in \mathfrak{g}$. In particular, $x\curlywedge_1 x=x\curlywedge_2 x=0$ in $\mathfrak{g}\curlywedge \mathfrak{g}$ and
$\overline{x}\curlywedge_1 \overline{x}=\overline{x}\curlywedge_2 \overline{x}=0$ in $\mathfrak{g}^{ab}\curlywedge \mathfrak{g}^{ab}$, where $\overline{x}$ denotes the coset $x+[\mathfrak{g}, \mathfrak{g}]$.
If we assume $x\notin [\mathfrak{g}, \mathfrak{g}]$, as in the proof of Theorem  \ref{T:t_e_centers}, we deduce that $\overline{x}\otimes \overline{x}\neq 0$ in  $\mathfrak{g}^{ab}\otimes \mathfrak{g}^{ab}$. But this contradicts to the natural isomorphism $\mathfrak{g}^{ab}\curlywedge \mathfrak{g}^{ab}\cong \mathfrak{g}^{ab}\otimes \mathfrak{g}^{ab} $, which holds by Lemma \ref{L:lemma ex ten}.
\end{proof}

\section{Capability of Lie algebras in the category of Leibniz algebras }\label{S:lie}

 As we mentioned above, a Lie algebra $L$ is capable if there exist a Lie algebra $Q$ such that
the quotient of $Q$ by its center is isomorphic to $L$. On the other hand, a Lia algebra $L$ can be viewed as a Leibniz algebra and one can
define the capability of $L$ in the category of Leibniz algebras. It is clear that if $L$ is capable in the category of Lie
algebras then $L$ is capable in the category of Leibniz algebras. The converse may not be true in general. The simplest example is
one dimensional Lie algebra, i.e. $\mathbb{K}$ equipped with the trivial Lie bracket. We show that there are other examples of Lie
algebras capable in the category of Leibniz algebras but not capable in the category of Lie algebras. First we prove the
following.

\begin{proposition}\label{P:5.1}
Let $\mathfrak{g}$ be a perfect Leibniz algebra. Then $\mathfrak{g}$ is capable if and only if $Z(\mathfrak{g})$ is trivial.
\end{proposition}
\begin{proof}
It suffices to show that $Z(\mathfrak{g})=Z^{\star}(\mathfrak{g})$. As we have seen, the implication
$Z^{\star}(\mathfrak{g})\subseteq Z(\mathfrak{g})$ is always true.  Show that if $\mathfrak{g}$ is perfect, then the reverse
implication $Z(\mathfrak{g})\subseteq Z^{\star}(\mathfrak{g})$ also holds. We have the natural homomorphism
 $Z(\mathfrak{g})\star \mathfrak{g} \to \mathfrak{g}\star \mathfrak{g}$, the image of which contains all the elements of the form
$g\star_1 x$ and $g\star_2 x$, where $g\in Z(\mathfrak{g})$ and $x\in \mathfrak{g}$. Thus, it suffices to show that
$Z(\mathfrak{g})\star \mathfrak{g}=0$. Since $g$ and $Z(\mathfrak{g})$ are acting trivially on each other, we have
\[
Z(\mathfrak{g})\star \mathfrak{g}=Z(\mathfrak{g})\star \mathfrak{g}^{ab} =0.
\]
\end{proof}

\begin{remark}
The analogue of Proposition \ref{P:5.1} in the category of Lie algebras can be proved in the same way.
\end{remark}

\begin{corollary}
Let $L$ be a perfect Lie algebra. Then $L$ is capable in the category of Lie algebras if and only if $L$ is capable in the
category of Leibniz algebras.
\end{corollary}

We need the following result.

\begin{proposition}\label{P:5.4}
Let $\mathfrak{g}$ be a perfect Leibniz algebra. Then for any Leibniz algebra $\mathfrak{h}$ we have:
\[
Z^{\star}(\mathfrak{g}\times \mathfrak{h}) = Z^{\star}(\mathfrak{g})\times Z^{\star}(\mathfrak{h}).
\]
\end{proposition}
\begin{proof}  We have natural homomorphisms $(\mathfrak{g}\times \mathfrak{h})\star (\mathfrak{g}\times \mathfrak{h}) \to
\mathfrak{g}\star \mathfrak{g}$ and  $(\mathfrak{g}\times \mathfrak{h})\star (\mathfrak{g}\times \mathfrak{h}) \to
\mathfrak{h}\star \mathfrak{h}$ implying that $Z^{\star}(\mathfrak{g}\times \mathfrak{h}) \subseteq
Z^{\star}(\mathfrak{g})\times Z^{\star}(\mathfrak{h})$. To prove the converse implication
$Z^{\star}(\mathfrak{g})\times Z^{\star}(\mathfrak{h})\subseteq Z^{\star}(\mathfrak{g}\times \mathfrak{h})$, it suffices to show
that the image of the following natural homomorphism $\big ( Z^{\star}(\mathfrak{g})\times Z^{\star}(\mathfrak{h})\big) \star
\big ( \mathfrak{g}\times \mathfrak{h}\big) \to \big ( \mathfrak{g}\times \mathfrak{h}\big) \star \big ( \mathfrak{g}\times
 \mathfrak{h}\big)$ is trivial. Since the underlying actions are trivial, we have
\begin{align*}
\big ( Z^{\star}(\mathfrak{g})\times Z^{\star}(\mathfrak{h})\big) \star\big ( \mathfrak{g}\times \mathfrak{h}\big)& =
\big ( Z^{\star}(\mathfrak{g})\times Z^{\star}(\mathfrak{h})\big) \star\big ( \mathfrak{g}\times \mathfrak{h}\big)^{ab}
= \big ( Z^{\star}(\mathfrak{g})\times Z^{\star}(\mathfrak{h})\big) \star  \big( 0\times \mathfrak{h}^{ab}\big)\\
&  =
\big ( Z^{\star}(\mathfrak{g})\times 0\big) \star  \big( 0\times \mathfrak{h}^{ab}\big)  \oplus
\big ( 0\times Z^{\star}(\mathfrak{h})\big) \star  \big( 0\times \mathfrak{h}^{ab}\big) \\
&= \big ( Z^{\star}(\mathfrak{g})\times 0\big) \star  \big( 0\times \mathfrak{h}\big) .
\end{align*}
Thus, it suffices to show that the image of the following natural homomorphism
$$\big ( Z^{\star}(\mathfrak{g})\times 0\big) \star  \big( 0\times \mathfrak{h}\big) \to  \big ( \mathfrak{g}\times \mathfrak{h}\big)
\star \big ( \mathfrak{g}\times  \mathfrak{h}\big)$$ is trivial. We have the following commutative diagram
 \begin{equation*}
\xymatrix@+10pt{
&\big ( Z^{\star}(\mathfrak{g})\times 0\big) \star  \big( 0\times \mathfrak{h}\big) \ar@{->}[rd]^{}
\ar@{->}[rr]^{ }
& &\big ( \mathfrak{g}\times 0\big) \star  \big( 0\times \mathfrak{h}\big) \ar@{->}[ld]_{}
 \\
& &  \big ( \mathfrak{g}\times \mathfrak{h}\big) \star \big ( \mathfrak{g}\times  \mathfrak{h}\big) &
}
\end{equation*}
where all arrows are defined naturally. Moreover,
\[
\big ( \mathfrak{g}\times 0\big) \star  \big( 0\times \mathfrak{h}\big) =
\big ( \mathfrak{g}\times 0\big)^{ab} \star  \big( 0\times \mathfrak{h}\big)^{ab} = 0 .
\]
\end{proof}

\begin{remark} The analogues of Proposition \ref{P:5.1} and Proposition \ref{P:5.4} in the category of groups
were proved by G. Donadze. We are not giving reference here because the results are not published.
\end{remark}

Given a Lie algebra $L$, assume that $Z^{\otimes}(L)$ (resp. $Z^{\wedge}(L)$) denotes the tensor (resp. the exterior) center of
$L$ in sense of \cite{Ni2}. Then, in the same way one can prove the following proposition.

\begin{proposition}\label{P:5.5}
Let $L$ be a perfect Lie algebra. Then for any Lie algebra $H$ we have:
\[
Z^{\otimes}(L\times H) =Z^{\otimes}(L)\times Z^{\otimes}(H) \quad \text{and} \quad
Z^{\wedge}(L\times H) =Z^{\wedge}(L)\times Z^{\wedge}(H).
\]
\end{proposition}

As a consequence we get the following result.

\begin{corollary}\label{C:last}
Let $L$ be a perfect Lie algebra. Then

(i) $L\times \mathbb{K}$ is not capable in the category of Lie algebras;

(ii)  $L\times \mathbb{K}$ is capable in the category of Leibniz algebras if and only if $Z(L)=0$.

\end{corollary}

\section*{Acknowledgements}
\noindent We thank Guram Donadze for usefull disscusions. Emzar Khmaladze is very grateful to the University of Santiago de Compostela for hospitality. He with Manuel Ladra was supported by Ministerio de Econom\'ia y Competitividad (Spain) (European FEDER support
included), grant MTM2016-79661-P.

\end{document}